\documentclass[11pt]{amsart}
\usepackage{graphicx}
\usepackage[active]{srcltx}
 \makeatletter
\renewcommand*\subjclass[2][2000]{%
  \def\@subjclass{#2}%
  \@ifundefined{subjclassname@#1}{%
    \ClassWarning{\@classname}{Unknown edition (#1) of Mathematics
      Subject Classification; using '1991'.}%
  }{%
    \@xp\let\@xp\subjclassname\csname subjclassname@#1\endcsname
  }%
}
 \makeatother
\usepackage{enumerate,amssymb,  mathrsfs,yhmath}

\newtheorem{theorem}{Theorem}[section]
\newtheorem{lemma}[theorem]{Lemma}
\newtheorem*{lemma*}{Lemma}

\def\1ton{1,2,\ldots,n}

\usepackage{amssymb}
\usepackage{amsthm}
\usepackage{mathrsfs, amsfonts, amsmath}
\usepackage{graphicx}

\theoremstyle{definition}

\theoremstyle{remark}
\newtheorem{remark}[theorem]{Remark}

\numberwithin{equation}{section}

\def\XXint#1#2#3{{\setbox0=\hbox{$#1{#2#3}{\int}$}
\vcenter{\hbox{$#2#3$}}\kern-.5\wd0}}

\def\ge{\geqslant}
\begin{document}

\title[On Riesz conjugate function theorem for harmonic functions]{ On generalized M. Riesz conjugate function theorem for harmonic mappings}

\author{Anton Gjokaj}
\address{Faculty of Natural Sciences and Mathematics, University of
Montenegro, Cetinjski put b.b. 81000 Podgorica, Montenegro}
\email{antondj@ucg.ac.me}
\author{David Kalaj}
\address{Faculty of Natural Sciences and Mathematics, University of
Montenegro, Cetinjski put b.b. 81000 Podgorica, Montenegro}
\email{davidk@ucg.ac.me}

\author{Djordjije Vujadinovi\'c}
\address{Faculty of Natural Sciences and Mathematics, University of
Montenegro, Cetinjski put b.b. 81000 Podgorica, Montenegro}
\email{djordjijevuj@ucg.ac.me}


\begin{abstract}
Let  $L^p(\mathbf{T})$ be the Lesbegue space of complex-valued functions defined in the unit circle $\mathbf{T}=\{z: |z|=1\}\subseteq \mathbb{C}$.
In this paper, we address the problem of finding the best constant
in  the inequality of the form:
$$\|(|P_+ f|^2+c| P_{-} f|^2)^{1/2}\|_{L^p(\mathbf{T})}\le A_{p,c}  \|f\|_{L^p(\mathbf{T})}.$$
Here $2\le p<\infty$, $c>0$, and by $P_{-} f$ and $ P_+ f$ are denoted co-analytic and analytic projection of a function $f\in L^p(\mathbf{T})$. The sharpness of the constant $A_{p,c}$ follows by taking a family quasiconformal harmonic mapping $f_\gamma$ and letting $\gamma\to 1/p$.
The result extends a sharp version of  M. Riesz conjugate function theorem of Pichorides and Verbitsky and some well-known estimates for holomorphic functions.
\end{abstract}
 \maketitle
\tableofcontents

\section{Introduction}

Let $\mathbf{U}$ denote the open unit disk and $\mathbf{T}$ the unit circle in the complex plane. For $p>1$, the Hardy space of harmonic functions, denoted by $\mathbf{h}^p$, consists of all harmonic mappings $f = g + \bar{h}$, where $g$ and $h$ are holomorphic functions on $\mathbf{U}$, such that
\[
\|f\|_p = \|f\|_{\mathbf{h}^p} = \sup_{0<r<1} M_p(f,r) < \infty,
\]
where
\[
M_p(f,r) = \left( \int_{\mathbf{T}} |f(r\zeta)|^p \, d\sigma(\zeta) \right)^{1/p},
\]
and $d\sigma(\zeta) = \frac{dt}{2\pi}$ for $\zeta = e^{it} \in \mathbf{T}$. The subspace of holomorphic functions in
$\mathbf{h}^p$
  is denoted by
$H^p$, and is known as the Hardy space.

If $f \in \mathbf{h}^p$, then it is a classical result that the radial limits
\[
f(e^{it}) = \lim_{r \to 1} f(re^{it})
\]
exist almost everywhere, and that $f \in L^p(\mathbf{T})$. Furthermore, the norm satisfies the identity
\begin{equation} \label{come}
\|f\|^p_{\mathbf{h}^p} = \lim_{r \to 1} \int_0^{2\pi} |f(re^{it})|^p \, \frac{dt}{2\pi} = \int_0^{2\pi} |f(e^{it})|^p \, \frac{dt}{2\pi}.
\end{equation}

Let $1 < p < \infty$ and define $\overline{p} = \max\{p, p/(p-1)\}$. In \cite{verb0}, Verbitsky established the following sharp inequalities. If $f = u + iv \in H^p$ with $v(0) = 0$, then
\begin{equation} \label{ver}
\sec\left( \frac{\pi}{2\overline{p}} \right) \|v\|_p \leq \|f\|_p,
\end{equation}
and
\begin{equation} \label{1ver}
\|f\|_p \leq \csc\left( \frac{\pi}{2\overline{p}} \right) \|u\|_p.
\end{equation}
These bounds improve the previously known result by Pichorides \cite{pik}, who showed
\begin{equation} \label{pico}
\|v\|_p \leq \cot\left( \frac{\pi}{2\overline{p}} \right) \|u\|_p.
\end{equation}
For further related work, see \cite{essen, studia, graf, verb2}. These inequalities were later generalized in \cite{tams} by Kalaj, where the author proved the following sharp bounds for harmonic mappings $f = g + \bar{h} \in \mathbf{h}^p$ with $\Re(h(0)g(0)) = 0$:
\begin{equation} \label{nes2}
\left( \int_{\mathbf{T}} \left( |g|^2 + |h|^2 \right)^{p/2} \right)^{1/p} \leq c_p \left( \int_{\mathbf{T}} |g + \bar{h}|^p \right)^{1/p},
\end{equation}
and
\begin{equation} \label{nes3}
\left( \int_{\mathbf{T}} |g + \bar{h}|^p \right)^{1/p} \leq d_p \left( \int_{\mathbf{T}} \left( |g|^2 + |h|^2 \right)^{p/2} \right)^{1/p},
\end{equation}
where
\[
c_p = \left( \sqrt{2} \sin\left( \frac{\pi}{2\bar{p}} \right) \right)^{-1}, \quad d_p = \sqrt{2} \cos\left( \frac{\pi}{2\bar{p}} \right).
\]
These bounds imply inequalities \eqref{pico}, \eqref{1ver}, and \eqref{ver}. As a consequence, Kalaj also verified the Hollenbeck-Verbitsky conjecture for the case $s = 2$ by applying inequality \eqref{nes2}.

The inequality \eqref{nes2} was later extended by Markovi\'c and Melentijevi\'c in \cite{melmar}, where they established the optimal constant $c_{p,s}$ in
\begin{equation} \label{nes4}
\left( \int_{\mathbf{T}} \left( |g|^s + |h|^s \right)^{p/s} \right)^{1/p} \leq c_{p,s} \left( \int_{\mathbf{T}} |g + \bar{h}|^p \right)^{1/p},
\end{equation}
for a specific range of $(p,s)$, including the case $(p,2)$. This result was further improved by Melentijevi\'c in \cite{mel}, who confirmed the Hollenbeck-Verbitsky conjecture for $s < \sec^2(\pi/(2p))$ when $p \leq 4/3$ or $p \geq 2$. However, the case $p \in [4/3, 2]$ remains unresolved.

In this work, we focus on a related problem. We aim to extend Verbitsky's result \eqref{ver} and the inequality \eqref{nes2} for the case $2 \leq p <\infty$. For $c > 0$ and $p \in [2, \infty)$, we determine the optimal constant in the inequality
\begin{equation} \label{nes32} \left( \int_{\mathbf{T}} \left( |g|^2 + c|h|^2 \right)^{p/2} \right)^{1/p} \leq A_{p,c} \left( \int_{\mathbf{T}} |g + \bar{h}|^p \right)^{1/p}.
\end{equation}
This estimate is equivalent to
\[
\left\| \left( |P_+[f]|^2 + c|P_-[f]|^2 \right)^{p/2} \right\|_p \leq a_{p,c}\|f\|_p ,
\]
where $f \in L^p(\mathbf{T})$ and $P_+$ and $P_-$ denote the analytic and co-analytic projections of $f$, respectively. For
\[
f(e^{it}) = \sum_{k=-\infty}^{\infty} c_k e^{ikt},
\]
we define
\[
P_+[f] = \sum_{k=0}^{\infty} c_k e^{ikt}, \quad P_-[f] = \sum_{k=1}^{\infty} c_{-k} e^{-ikt},
\]
with Fourier coefficients
\[
c_k = \frac{1}{2\pi} \int_0^{2\pi} f(e^{it}) e^{-ikt} \, dt.
\]
We observe that the inequalities \eqref{ver}, \eqref{pico}, \eqref{nes2}, \eqref{nes3}, and \eqref{nes4} can all be reformulated in terms of these projections.

\begin{theorem}\label{Kao}
Let $ p\in [2,\infty)$ and assume that $c>0$.
Then we have the following sharp inequality
\begin{equation}\label{nesit}\left(\int_{\mathbf{T}}(|g|^2+c  |h|^2)^{\frac{p}{2}}\right)^{1/p}\le a_{p,c}\left(\int_{\mathbf{T}} |g+\bar h|^p\right)^{1/p},\end{equation}
for $f=g+\bar h\in \mathbf{h}^p$ with $\Re(g(0)h(0))=  0$ and
\begin{equation}\label{apc}a_{p,c}=  2^{-1/2} \left(\left(1+c+\sqrt{1+c^2+2 c \cos\left[\frac{2 \pi }{p}\right]}\right) \csc^2\left[\frac{\pi }{p}\right]\right)^{1/2}.\end{equation}
The equality is never attained (except for a zero function $f$). However there is a minimizing sequence converging to a quasiconformal harmonic mapping $f$ provided that $c\neq 1$. If $c=1$ then the minimizer is a real harmonic function. In both cases, the limiting mapping is not in $\mathbf{h}^p$.
\end{theorem}
\begin{remark} Instead of the condition $\Re(g(0)h(0))=  0$ we can in  Theorem~\ref{Kao} consider more general condition for a given \( p \), the argument \(\theta = \arg(g(0) h(0))\) must satisfy the following conditions:

\begin{equation}\label{condit}
\text{for } p \in [2,4]: \quad
\max\left\{0,\, \frac{\pi(p - 3)}{p} \right\} \leq |\theta| \leq \frac{\pi(p - 1)}{p},
\end{equation}

and

\begin{equation}\label{condit1}
\text{for } p \geq 4: \quad
\frac{\pi}{p} \leq |\theta| \leq \pi - \frac{\pi}{p}.
\end{equation}
If we drop the conditions \eqref{condit} and \eqref{condit1}, then our main inequality may fail. Namely for $c=1$, $g=h=\alpha=e^{is}$ with $$\frac{\pi(p - 1)}{2p}<s<\pi,$$ the inequality \eqref{nesit} becomes 
$$-2^{p/2} +  2^{p/2} |\cos[s]|^p \csc\left[\frac{\pi}{2 p}\right]^p\ge 0$$ which is not true if $s\in[\frac{\pi(p - 1)}{2p},\frac{\pi}{2}]$.

For $c=1$, \eqref{nesit} becomes \eqref{nes2} and for  $g=h$ and $c=1$ it becomes  \eqref{1ver}.
\end{remark}

Let $p\ge 2$ and $c>0$.  The main task in the proof of Theorem~\ref{Kao} is to find optimal positive constants $a_{p,c}$, $b_{p,c}$,  and pluri-subharmonic functions $\mathcal{G}_p(z,w)$ (Lemma~\ref{lemsub}) for $z,w\in \mathbf{C}$, vanishing for $z=0$ or $w=0$, so that the inequality
\begin{equation}\label{calf}
 (|w|^2+ c |z|^2)^{\frac{p}{2}} \le a_{p,c} |w+\bar z|^p- b_{p,c} \mathcal{G}_p(z,w)
\end{equation}
 is sharp.

\begin{lemma}\label{medium}
Let \( 2 \le p \le 4 \), and define
\[
h(t) =h_p(t):= -\cos\left( \frac{p}{2}(\pi - |t|) \right), \quad t \in [-\pi, \pi].
\]
Extend \( h \) to the interval \( [-2\pi, 2\pi] \) by setting
\[
h(t) := h(|t| - \pi), \quad \text{for } \pi \le |t| \le 2\pi.
\]
For \( p > 4 \), define \( h(t) \) for \( 0 \le t \le \pi \) by
\[
h(t) =
\begin{cases}
-\cos\left[\frac{1}{4} p \left(\pi - \left| \pi - 2t \right|\right)\right], & t \in \left[0,\frac{2\pi}{p}\right] \cup \left[\pi - \frac{2\pi}{p}, \pi\right], \\[10pt]
\max\left( \left| \cos\left[\frac{1}{2} p (\pi - t)\right] \right|,\ \left| \cos\left[\frac{p t}{2}\right] \right| \right), & t \in \left(\frac{2\pi}{p},  \pi - \frac{2\pi}{p}\right).
\end{cases}
\]

Extend \( h(t) \) to \( [-\pi, \pi] \) by defining
\[
h(-t) := h(t).
\]

For \( \pi \le |t| \le 2\pi \), extend by symmetry:
\[
h(t) := h(2\pi - |t|).
\]

Let $a_{p,c}$ be defined in \eqref{apc} and assume \begin{equation}\label{bpc} b_{p,c}=  (1 + c + S)  \csc\left(\frac{\pi}{p}\right) \left(S \cdot \sec\left(\frac{\pi}{p}\right)\right)^{\frac{p}{2} - 1},
\end{equation}
where \begin{equation}\label{sss} S = \sqrt{1 + c^2 + 2c \cos\left(\frac{2\pi}{p}\right)}.\end{equation}
Then for complex numbers $z$ and $w$, and \( 2 \le p <\infty \), $-\pi\le s,t\le \pi$ we have:
\[
(|z|^2 + c |w|^2)^{p/2} \le a_{p,c} |z + \bar{w}|^p - b_{p,c} (|zw|)^{p/2} h_p(s + t).
\]

Equality is attained when \( |z| = R |w| \) and \( \arg(zw) = \pi - \frac{\pi}{p} \), where $$R=\frac{2c \cos \frac{\pi}{p}}{-1+c+\sqrt{1+c^2+2c\cos(2\pi/p)}}.$$
\end{lemma}


Lemma follow from the following theorem, whose proof is very technical and involved.

\begin{theorem}\label{theo123}
For every $p \in [2,\infty)$, $r > 0$, $t \in [-2\pi, 2\pi]$, and $c > 0$, we have
\begin{equation}\label{ourineq}
      A \left(\frac{1 + r^2 + 2r \cos t}{2r} \right)^{p/2} -B h(t)- \left(\frac{c + r^2}{2r}\right)^{p/2} \geq 0,
\end{equation}
where
\begin{align*}
A &= 2^{-p/2} \left((1 + c + S)  \csc^2\left(\frac{\pi}{p}\right)\right)^{p/2}, \\
B &= 2^{-p/2} (1 + c + S)  \csc\left(\frac{\pi}{p}\right) \cdot \left(S  \sec\left(\frac{\pi}{p}\right)\right)^{\frac{p}{2} - 1},
\end{align*} and $S$ is as in \eqref{sss}.
Equality holds for $t = \pi - \pi/p$ and
\[
    r = R := \frac{2c \cos \frac{\pi}{p}}{-1+c+\sqrt{1+c^2+2c\cos(2\pi/p)}}.
\]
\end{theorem}
\begin{lemma}\label{lemsub}
For $2\le p\le 4$, $z=|z|e^{i\theta}$ the function $$\Phi_p(z)= -|z|^{p/2}\cos\frac{p}{2}(\pi-|\theta|),
$$ is subharmonic in $\mathbf{C}$ and $\mathcal{G}_p(z,w)=\Phi_p(zw)$ is pluri-subharmonic in $\mathbb{C}^2$.
\end{lemma}
\begin{proof}[Proof of Lemma~\ref{lemsub}]
Let $2\le p\le 4$. $z_0=re^{i\theta}\in \mathbf{C}\setminus\{0\}$. If $\theta=0$, then near $z_0$, $\Phi_p(z)=\max\{-|z|^{p/2}\cos\frac{p}{2}(\pi-\theta), -|z|^{p/2}\cos\frac{p}{2}(\pi+\theta)\}$. Since $F_+(z)=-|z|^{p/2}\cos\frac{p}{2}(\pi-\theta)$ and $F_{-}=-|z|^{p/2}\cos\frac{p}{2}(\pi+\theta)$ are localy harmonic, it follows that $\Phi_p$ is subharmonic near $z_0$. If $\theta\neq 0$, then $\Phi_p$ coincides with $F_+$ or $F_{-}$ near $z_0$. Finally, since $$\frac{1}{2\pi r}\int_{|z|=r}\Phi_p(z) |dz|=-\frac{4 r^{p/2} \sin  \left[\frac{p \pi }{2}\right]}{ p}\ge 0=\Phi_p(0),$$ we obtain that $\Phi_p$ is subharmonic in $\mathbf{C}$. 
\end{proof}
\begin{lemma}\label{lemsub2}\cite[Lemma~3]{verb0},\cite{verb2}.
For $p\ge 4$, $z=|z|e^{i\theta}$ the function $$\Psi_p(z)=|z|^{p/2}h_p(\theta)$$ is subharmonic in $\mathbf{C}$ and the function $\mathcal{G}_p(z,w)=\Psi_p(zw)$ is pluri-subharmonic in $\mathbb{C}^2$.
\end{lemma}
\begin{proof}[Proof of Lemma~\ref{lemsub2}]
The function $\Psi_p(z)$ coincides with $\Phi_{p/2}(z)$ from the paper \cite{verb2}. For the completeness include its proof.
Let $z_0=re^{i\theta}\in \mathbf{C}\setminus\{0\}$. If $\theta\neq 0$,  then near $z_0$,  $\Psi_p$ coincides with a harmonic function, and so is subharmonic in $z_0$. If $\theta=0$, then $\Psi_p$ is equal to the maximum of several harmonic functions of the form $u(re^{i\theta}) = r^{p/2}\cos(p/2(\theta_0 +\theta))$.
Finally, since
$h_p(\pi/p-x)=-h(\pi/p+x)$ for $x\in[\pi/2-2\pi/p,\pi/2]$ and $h_p(\pi/2-x)=h(\pi/2+x)$, we obtain $$ \int_{\pi/2-2\pi/p}^{\pi/2+2\pi/p}h_p(x)dx=0.$$ Thus
$$\frac{1}{2\pi r}\int_{|z|=r}F(z) |dz|=\frac{2r^{p/2}}{2\pi r}\int_{0}^{\pi/2-2\pi/p}h_p(x) dx> 0=F(0),$$ we obtain that $F$ is subharmonic in $\mathbf{C}$.
\end{proof}

\section{Proof of Theorem~\ref{Kao}}
\subsection{Proof of the inequality statement}
\begin{proof}[Proof of inequality of Theorem~\ref{Kao}]

We use  Lemma~\ref{medium}. Assume that $f=g+\bar h$, where $g$ and $h$ are holomorphic functions on the unit disk. Then from Lemma~\ref{medium}, we have $$ (|g(z)|^2+c |h(z)|^2)^{\frac{p}{2}} \le a_{p,c}|g(z)+\overline{h(z)}|^p  - b_{p,c} \mathcal{G}_p(g(z),h(z)),$$ where
$a_{p,c}$ and $b_{p,c}$ are given in \eqref{apc} and \eqref{bpc}.
Then $$\int_{\mathbf{T}} (|g(z)|^2+c |h(z)|^2)^{\frac{p}{2}}\le a_{p,c} \int_{\mathbf{T}} |g(z)+\overline{h(z)}|^p - b_{p,c} \int_{\mathbf{T}}\mathcal{G}_p(g(z),h(z)).$$
Let $\theta=\arg(g(0)h(0))$. As $\mathcal{P}_p(z)= \mathcal{G}_p(g(z),h(z))$ is subharmonic for every $p\ge 2$, by sub-mean inequality we have that $$\int_{\mathbf{T}}\mathcal{G}_p(g(z),h(z))\ge \mathcal{G}_p(g(0),h(0))=-|g(0)h(0)|^ph_p(\theta)$$ where for $p\in[2,4]$, $$h_p(\theta)=-\cos\left(p\frac{\pi-|\theta|}{2}\right)\ge 0,$$ because $\Re(g(0)h(0))=0$, which means that $\theta=\pm \frac{\pi}{2}$ or  $\theta$ satisfies the condition \eqref{condit} respectively \eqref{condit1}. Similarly we treat the case $p\ge 4$.
\end{proof}
\subsection{Proof of the sharpness case}

Let \( g(z) = \left( \frac{1 - z}{1 + z} \right)^\gamma \) where \( \gamma > 0 \), \( z \in \mathbf{T} \), \( a \in \mathbb{C} \), and \( c \in \mathbb{R}_+ \). Consider the two integrals:

1. \( \int_{\mathbf{T}} |f(z)|^p \, |dz| \) where \( f(z) = g(z) +  \overline{a g(z)} \)

2. \( \int_{\mathbf{T}} \left( c|g(z)|^2 +  |a|^2 |g(z)|^2 \right)^{p/2} \, |dz| \)

On the unit circle \( \mathbf{T} \), we have \( \overline{g(z)} = (-1)^\gamma g(z) \), so:

\[
f(z) = \left( 1 + a(-1)^\gamma \right) g(z)
\]

Thus:

\[
|f(z)|^p = |1 + a(-1)^\gamma|^p \cdot |g(z)|^p
\]

The first integral is:

\[
\int_{\mathbf{T}} |f(z)|^p \, |dz| = |1 + a(-1)^\gamma|^p \cdot \int_{\mathbf{T}} |g(z)|^p \, |dz| = 4 |1 + a(-1)^\gamma|^p \int_0^{\pi/2} \tan^{pc} u \, du
\]

We now compute the second integral:

\[
\int_{\mathbf{T}} \left( c|g(z)|^2 +  |a|^2 |g(z)|^2 \right)^{p/2} \, |dz| = \int_{\mathbf{T}} \left( c + |a|^2 \right)^{p/2} |g(z)|^p \, |dz|.
\]

This simplifies to:

\[
= (c + |a|^2)^{p/2} \cdot \int_{\mathbf{T}} |g(z)|^p \, |dz| = 4 (c + |a|^2)^{p/2} \int_0^{\pi/2} \tan^{pc} u \, du.
\]

The ratio between the two integrals is:

\[
C^p_{p,c,\gamma}:=\frac{
\int_{\mathbf{T}} \left(c |g(z)|^2 + |a|^2 |g(z)|^2 \right)^{p/2} |dz|
}{
\int_{\mathbf{T}} |f(z)|^p \, |dz|
}
=
\frac{
(c + |a|^2)^{p/2}
}{
|1 + a(-1)^\gamma|^p
}.
\]

Let $q=\pi/p$ and \[
r = \frac{2c \cos\left( {q} \right)}{-1 + c + \sqrt{1 + c^2 + 2c \cos\left(2q\right)}}.
\]
Now if $a=-re^{-\frac{2\pi}{p} \imath}$, when $\gamma\to 1/p$ the constant $C^p_{p,c,\gamma}$ tends to

$$C^p_{p,c}= 2^{-p/2} \left(\left(1+c+\sqrt{1+c^2+2 c \cos (2q)}\right) \csc^2 q\right)^{p/2}.$$

Observe also that in this case $g(0)h(0) =e^{\imath(\pi-2\pi/p)}$, $\theta=\pi-2\pi/p$ which means that it is satisfied the condition \eqref{condit} and \eqref{condit1}, respectively. Moreover  $-\cos p/2(\pi-\theta)=1$. Observe that for $c\in(0,1)$ we have that  $r>1$ and $f$ is $1/r-$ quasiconformal harmonic and for $c>1$, $0<r<1$ and $f$ is $r$-quasiconformal harmonic. For $c=1$, $f$ is real function up to a complex factor. 
\section{Proof of Theorem~\ref{theo123}}
\begin{proof}[Proof of Theorem~\ref{theo123}]
Since \(\cos(-t) = \cos t\) and \(h(-t) = h(t)\) for \(t \in [-\pi, \pi]\), it suffices to prove the inequality for \(t \in [0, \pi]\). Similarly, because \(\cos(2\pi - t) = \cos t\) and \(h(2\pi - t) = h(t)\), the interval \(t \in [\pi, 2\pi]\) can also be reduced to the case \(t \in [0, \pi]\). In the same way, for \(t \in [-2\pi, -\pi]\), we have \(\cos(2\pi + t) = \cos t\) and \(h(2\pi + t) = h(t)\), so this interval reduces to \(t \in [-\pi, 0]\), which is again covered by the symmetry of the function.

Therefore, by these symmetries, it is sufficient to prove the inequality for \(t \in [0, \pi]\).

Let us use shorthand notation \begin{equation}\label{pq}q=\frac{\pi}{p}.\end{equation}
Expressing $c$ as a function of $R$:
\[
    c = \frac{R (1-R \cos q)}{R-\cos q}.
\]
Since $c > 0$, it follows that
\[
    \cos q < R < \sec q.
\]

We then obtain
\begin{align*}
    A &= \left(\frac{R}{R-\cos q}\right)^{p/2}, \\
    B &= 2^{1-p/2} R (R - \cos q)^{-p/2} \left(1+R^2-2R \cos q\right)^{(p-2)/2} \sin(\pi/p).
\end{align*}
Then our inequality becomes \begin{equation}\label{bein}\begin{split}\left(\frac{1}{2r}+\frac{r}{2}+ \cos t\right)^{p/2}&-\left(\frac{\left(1+r^2\right) R-\left(r^2+R^2\right) \cos q}{2r R}\right)^{p/2}\\&+  \left(\frac{1+R^2-2 R \cos q}{2R}\right)^{-1+\frac{p}{2}}\cos \left[\frac{1}{2} p (\pi -t)\right] \sin q\ge 0.\end{split}\end{equation}
The minimum value is zero and is attained for $t = \pi - \pi/p$ and $r = R$.
\subsection{Proof of the case $p\in[2,4]$}

This section contains the proof of the case when $h(t) =-\cos(\frac{p}{2}(\pi-t)$ for $t\in[0,\pi]$ and $p\ge 2$ arbitrary.

Introducing the transformations $0\le a\le b$:
\[
    2\cosh a = R + \frac{1}{R}, \quad 2\cosh b = r + \frac{1}{r},
\]
our inequality \eqref{ourineq} becomes
\[\begin{split}
    -(\cosh b &- \cos q \cosh(a \pm b))^{p/2}+ (\cosh b + \cos t)^{p/2} \\
   & + \left(\cosh a - \cos q\right)^{-1 + p/2} \cos\left(\frac{p}{2} (\pi - t)\right) \sin q \geq 0,\end{split}
\] where the equality is attained for $a-b$ instead of $a\pm b$ and  if $a=b$ and $t=\pi -\pi/p$.

Since
$
    \cos q < R < \sec q,
$
we have the constraint
\begin{equation}\label{constr}
    \cosh a \leq \frac{1}{2}\left(\cos q + \sec q\right).
\end{equation}

Thus, it suffices to prove that for $x \in [0,\pi]$, and $b \geq a \geq 0$, after using the change
 $t= \pi - x$,
the following inequality holds under the given constraints.

\begin{align*}
    H(x) := & -\left(\cosh b - \cos q  \cosh(b - a)\right)^{p/2} + (\cosh b - \cos x)^{p/2} \\
    & + \left(\cosh a - \cos q \right)^{-1 + \frac{p}{2}} \cos\left(\frac{p x}{2}\right) \sin q  \geq 0.
\end{align*}
Equality is attained for $a=b$ and $x=\pi/p$.

We continue  by the following lemma
\begin{lemma}\label{lemaimpo}
For  \(x\in[0,\pi]\) we have the following inequality:
\begin{equation}\label{thisin}
\cos {q} - \cos x + \frac{2 \cos \frac{p x}{2} \sin {q}}{p} - \cos^2 \frac{p x}{2} \left(-1 + \cos {q} + \frac{2 \sin {q}}{p}\right) \geq 0.
\end{equation}
\end{lemma}
\begin{proof}
Then, equation~\eqref{thisin} is equivalent to verifying the inequality:
\[
H(y) := \cos q - \cos\left( \frac{2 \arccos y}{p} \right) + \frac{2 y \sin q}{p} - y^2 \left(-1 + \cos q + \frac{2 \sin q}{p}\right) \geq 0,
\]
where \( y = \cos\left( \frac{p x}{2} \right) \in [-1,1] \), and \( x \in \left[0, \frac{2\pi}{p} \right] \).

If \( x \ge \frac{2\pi}{p} \), then there exists \( x' \in \left[0, \frac{2\pi}{p} \right] \) such that
\[
\cos\left( \frac{p x'}{2} \right) = \cos\left( \frac{p x}{2} \right),
\]
which implies \( -\cos x \ge -\cos x' \). Therefore, proving the inequality for \( x' \) suffices.

Now, letting \( y = \cos s \), we compute the third derivative of \( H \) with respect to \( y \):
\[
H'''(\cos s) = \frac{ \csc^5 s \left( 6p \cos\left( \frac{2s}{p} \right) \sin(2s) - 2 \left( -4 + p^2 + 2(2 + p^2) \cos^2 s \right) \sin\left( \frac{2s}{p} \right) \right) }{p^3}.
\]

To continue, we first prove the following lemma.
\begin{lemma}\label{leka}
For \(s\in[0,\pi]\), \(H'''(\cos s)\leq 0\).
\end{lemma}

\begin{proof}[Proof of Lemma~\ref{leka}]

Define
\[
\Phi(s) = 6p \cos\left( \frac{2s}{p} \right) \sin(2s)
- 2\left( -4 + p^2 + 2(2 + p^2) \cos^2 s \right) \sin\left( \frac{2s}{p} \right).
\]

Then we have
\[
\Phi'(s) = \frac{4(-1 + p^2)}{p} \left( -4 \cos\left( \frac{2s}{p} \right) \sin^2 s
+ p \sin(2s) \sin\left( \frac{2s}{p} \right) \right).
\]

We aim to prove that \( \Phi'(s) \le 0 \).
Since \( \sin s \cdot \sin\left( \frac{2s}{p} \right) \ge 0 \), the inequality is equivalent to showing:
\[
\phi(s) := \cot s - \frac{2}{p} \cot\left( \frac{2s}{p} \right) \le 0.
\]

Now we compute the derivative of \( \phi \):
\[
\phi'(s) = \frac{\csc^2\left( \frac{2s}{p} \right) \left( -4s + p \sin\left( \frac{4s}{p} \right) \right)}{p^3} \le 0,
\]
because \( \sin y \le y \) for every \( y \ge 0 \).
Thus, \( \phi(s) \le \lim_{s \to 0} \phi(s) = 0 \), which implies
\[
\Phi'(s) \le 0 \quad \text{and hence} \quad \Phi(s) \le \Phi(0) = 0.
\]

\end{proof}
Since  \(H'''(y)\leq 0\), we get that \(H'\) is concave. As \(H'(0)=0\) and
\[
H'(1) =2-\frac{4}{p^2}-2 \cos \left[\frac{\pi }{p}\right]-\frac{2 \sin \left[\frac{\pi }{p}\right]}{p} < 0
\] and   \(H(0) = H(\cos{q}) = 0\)
there exists exactly one point \(x_\circ\in(0,1)\) so that \(H'(x_\circ)=0\). Then \(H\) is increasing in \([0,x_\circ]\) and decreasing in \([x_\circ, 1]\).  It follows that \(H(x) \geq 0\) for every \(x\). This finishes the proof of Lemm~\ref{lemaimpo}.
\end{proof}

Now in view of \eqref{thisin},  our inequality reduces to:
\begin{equation}
\begin{split}
-\left( \cosh b - \cos q \cosh (b - a) \right)^{p/2}
&+ \left( \cosh b - \cos q - \frac{2 \cos \frac{p x}{2} \sin q}{p} \right. \\
&\quad + \cos^2 \frac{p x}{2} \left(-1+\cos q+\frac{2 \sin q}{p} \right) \bigg)^{p/2} \\
&+ \left( \cosh a - \cos q \right)^{-1+\frac{p}{2}}
\cos \frac{p x}{2} \sin q \geq 0.
\end{split}
\end{equation}

Equivalently, we write:
\begin{equation}
\begin{split}
\cosh b - \cos q - \frac{2 \cos \frac{p x}{2} \sin q}{p}
&+ \cos^2 \frac{p x}{2} \left(-1+\cos q+\frac{2 \sin q}{p}\right) \\
&\geq \left( \left( \cosh b - \cos q \cosh (b-a) \right)^{p/2} \right. \\
&\quad - \left( \cosh a - \cos q \right)^{p/2-1} \cos \frac{p x}{2} \sin q \bigg)^{2/p}.
\end{split}
\end{equation}

Now, let \( y = \cos \frac{p x}{2} \), and define
\begin{equation*}
k = \frac{(\cosh a - \cos q)^{-1+\frac{p}{2}} \sin q}
{\left(\cosh b - \cos q \cosh (b - a)\right)^{p/2}}.
\end{equation*} Then the inequality becomes
\begin{equation}
\begin{split}
\cosh b - \cos q - \frac{2y \sin q}{p}
&+ y^2 \left(-1 + \cos q + \frac{2\sin q}{p}\right) \\
&\geq \left( \cosh b - \cos q \cosh (b-a) \right) \\
&\quad \times \left( 1 -ky \right)^{2/p}.
\end{split}
\end{equation}
Now introduce  the change of variables $(u,v)$:
\begin{equation*}
-\cos q \cosh (a-b) + \cosh b = u (1-\cos {q}),
\end{equation*}
\begin{equation*}
\cosh a - \cos q = v (1-\cos {q}),
\end{equation*}
and notice that $u,v\ge 1$, and in view of \eqref{constr}, $v\le  \frac{1}{2} \left(1+\sec q\right)$ and the inequality is an equality precisely when $u=v=1$.
Then we obtain:
\begin{equation*}
k = \frac{\left(\frac{v}{u}\right)^{p/2}  \cot \frac{\pi }{2 p}}{v}.
\end{equation*}

Our inequality is then equivalent to:
\begin{equation}\label{ejte2}
\begin{split}
\cosh b - \cos q - \frac{2 y \sin q}{p}
&+ y^2 \left(-1+\cos q+\frac{2 \sin q}{p}\right) \\
&\geq u (1-\cos q) \left(1-\frac{\left(\frac{v}{u}\right)^{p/2} y \cot \frac{\pi }{2 p}}{v}\right)^{2/p}.
\end{split}
\end{equation}


\textit{Claim:} for $y\in[0,1/k]$, \begin{equation}\label{secondquad}1-\frac{2 k y}{p}-\frac{k^2 (p-2) y^2}{p^2}\ge (1-k y)^{2/p}.\end{equation}

To prove the claim, observe that for $$\phi(y)=-1+(1-k y)^{2/p}+\frac{2 k y}{p}+\frac{k^2 (p-2) y^2}{p^2},$$ $$\phi'''(y)= -\frac{4 k^3 (p-2) (-1+p) (1-k y)^{-3+\frac{2}{p}}}{p^3}\le 0$$ and thus $\phi''$ is decreasing. As $\phi''(0)=0$, it follows that $\phi''(y)<0$ and thus $\phi'$ is decreasing. Finally, $\phi'(0)=0$ and thus $\phi(y)<0$. Thus $\phi$ is decreasing which mean $\phi(y)<\phi(0)$ which coincides with \eqref{secondquad}.

Then in view of  \eqref{secondquad}, inequality \eqref{ejte2} will follow from the following quadratic inequality  on $y$ \begin{equation}
\label{firstquad}
\begin{split}
    \cosh(b) &- \cos\left({q}\right) - \frac{2y \sin\left({q}\right)}{p}
    + y^2 \left(-1 + \cos\left({q}\right) + \frac{2 \sin\left({q}\right)}{p} \right) \\
    &\geq u \left(1 - \cos\left({q}\right)\right)
    \left(1 - \frac{2 k y}{p} - \frac{k^2 (p-2) y^2}{p^2}\right).
\end{split}
\end{equation}

Since \[
\cosh b = \frac{M + N}{K},
\]
where:
\[
\begin{aligned}
M &= 2 \sqrt{(-1 + v) \cos^2 q \left(1 - u^2 + (1 + u^2 - 2v) \cos q\right)\left(-1 - v + (-1 + v) \cos q\right)}, \\
N &= u + uv - 2uv \cos q + u(-1 + v) \cos 2q, \\
K &= 2 + (2 - 4v) \cos q.
\end{aligned}
\]

Then the previous inequality can be written as $$a_2y^2+a_1y + a_0\ge 0,$$ where

\begin{align*}
a_0 &= \frac{\cos(q) \left(U + \sqrt{V}\right)}{1 + (1 - 2v) \cos(q)} \\[1ex]
a_1 &= -\frac{2 \left(v - u \left(\frac{v}{u}\right)^{p/2} \right) \sin(q)}{p v} \\[1ex]
a_2 &= \frac{(p-2) u \left(\frac{v}{u}\right)^p \left(1 + \cos(q)\right)
+ p v^2 \left(p \left(-1 + \cos(q)\right) + 2 \sin(q)\right)}{p^2 v^2}.
\end{align*}

Here $$U=1 + 2u - 3uv + (1 - 2v + u(-2 + 3v)) \cos\left({q}\right) $$ and $$V=(v - 1) \left(1 - u^2 + (1 + u^2 - 2v) \cos\left({q}\right)\right) \left(-1 - v + (-1 + v) \cos\left({q}\right)\right).$$
Our goal is to show that \begin{equation}\label{discrim}a_1^2-4 a_0 a_2\le 0.\end{equation}
We divide the proof into two cases\\
\textbf{{(i) Proof of the case $u\ge v$.}}

\begin{lemma}\label{lemapm}
Let $$g(u,v)=-\cos\left({q}\right) \left(U+ \sqrt{V}\right)$$

Then $g(u,v)\ge g(v,v)\ge 0$ for $1\le v\le w$ and $u\ge v$, for $$w=\frac{1}{2} \left(1+\sec q\right).$$ Moreover if $h(v)=\frac{g(u,v)-g(v,v)}{u-v}$, then $h(v)\ge h(1)= \cos{q}\left(1-\cos {q}\right)$.
\end{lemma}
\begin{proof}[Proof of Lemma~\ref{lemapm}]

We first have $$g(v,v)=2 (-1+v) v \left(1-\cos q\right) \cos q$$ and then this function is positive for $v\in [1,w]$. Now we prove that $k(u,v)=g(u,v)-g(v,v)\ge 0$.

Then  $$k(u,v)=\cos q \left(W-\sqrt{V}\right)$$

where $$W=-1-2 (-1+v) v+u (-2+3 v)+\left(-1+u (2-3 v)+2 v^2\right) \cos q$$ and $$V=(-1+v) \left(1-u^2+\left(1+u^2-2 v\right) \cos q\right) \left(-1-v+(-1+v) \cos q\right).$$

Then by direct computation we find that $$h(v) -h(1)=\cos{q} \frac{T+\sqrt{V}}{u-v} $$
where $$T=(1+3 u-2 v) (-1+v)+(-1+v) (1-3 u+2 v) \cos q$$

Then it is clear that $T\ge 0$ and thus $h(v)\ge h(1)=\cos{q}\left(1-\cos q\right)$.
\end{proof}
We need the following lemma
\begin{lemma}\label{shtune} For $p\in[2,\infty)$ and $b\in(0,1)$ we have
$$\psi(p)=\frac{4 \left(b-b^{p/2}\right)^2}{(1-b)^2 b (p-2)^2}\le 1.$$
\end{lemma}
\begin{proof}[Proof of Lemma~\ref{shtune}]
Observe first that $$\psi'(p)=\frac{4 \left(-b+b^{p/2}\right) \left(2 b+b^{p/2} (-2+(p-2) \log b)\right)}{(-1+b)^2 b (p-2)^3}.$$ Prove that $\psi'(p)<0$. Is is equivalent with the inequality $$\phi(p)=2 b+b^{p/2} (-2+(p-2) \log b)\ge 0.$$ Since $$\phi'(p)=\frac{1}{2} b^{p/2} (p-2) \log^2 b>0$$ it follows that $\phi$ is increasing in $[2,\infty)$. Thus $\phi(p)\ge \phi(2)=0$. Thus $\psi$ is decreasing, and we need to prove that $$\lim_{p\to 2}\psi(p)=\frac{b \log^2 b}{(-1+b)^2}\le 1.$$ The last inequality after the change $b=e^{-2c}$, with $c\ge 0$ is equivalent with $c\le\sinh c$.  This completes the proof of lemma.
\end{proof}
Recall that  for $$h(v)=\frac{g(u,v)-g(v,v)}{u-v},$$ where $g$ is defined in Lemma~\ref{lemapm} we proved the inequality $$h(v)\le h(1)=\left(1-\cos q\right) \cos q.$$
Then to prove the inequality
$a_1^2-4 a_0 a_2\le 0$ it is enough to prove \[
\frac{4 \left(v - u \left(\frac{v}{u}\right)^{p/2}\right)^2 \sin^2 q}{p^2 (u - v) v^2}
$$ $$- \frac{4 (1 - \cos q)\cos q \left[(p-2) u \left(\frac{v}{u}\right)^p (1 + \cos q) + p v^2 \left(p (-1 + \cos q) + 2 \sin q\right)\right]}{p^2 v^2 \left(1 + (1 - 2v)\cos q\right)}
\le 0.
\]

By Lemma~\ref{shtune}, with $b=v/u$, it is enough to prove:

$$\frac{(p-2)^2 (u-v) \sin q^2}{p^2 u v}$$ $$+\frac{4 u^{-p} \cos q \left((p-2) u v^p \sin^2 q-p u^p v^2 \left(-1+\cos q\right) \left(p \left(-1+\cos q\right)+2 \sin q\right)\right)}{p^2 v^2 \left(-1+(-1+2 v) \cos q\right)}\le 0$$

or what is the same: $$\frac{(p-2)^2 (u-v) v \sin^2 q}{u}$$ $$+\frac{4 u^{-p} \cos q \left((p-2) u v^p \sin^2 q-p u^p v^2 \left(-1+\cos q\right) \left(p \left(-1+\cos q\right)+2 \sin q\right)\right)}{-1+(-1+2 v) \cos q}\le 0.$$

By using the change $u=v/b$, it reduces to $$K(b)=-(-1+b) (p-2)^2 v \sin^2 q$$ $$+\frac{8 v \cos q \sin^2 \frac{q}{2} \left(b^p (p-2) \left(1+\cos q\right)+b p v \left(p \left(-1+\cos q\right)+2 \sin q\right)\right)}{b \left(-1+(-1+2 v) \cos q\right)}\le 0.$$

 Now $$K'(b)= \frac{8 b^{-2+p} (p-2) (-1+p) v \cos q \left(1+\cos q\right) \sin^2 \frac{q}{2}}{-1+(-1+2 v) \cos q}-(p-2)^2 v \sin^2 q\le 0$$ for $v\in[1,1+  \sec q \sin^2 \frac{\pi }{2p}]$  and thus $K$ is decreasing. Since $$K(1)=\frac{8 v \cos q \sin^2 \frac{q}{2} \left((p-2) \left(1+\cos q\right)+p v \left(p \left(-1+\cos q\right)+2 \sin q\right)\right)}{-1+(-1+2 v) \cos q}\le 0$$ for $p\in[2,\infty)$ and $v\in [1,1+\sec q \sin^2 \frac{\pi }{2p})$, the inequality follows.
\\
\textbf{(ii) Proof of the case $u<v$}

We make use the new parameters \( t, s \in [0,1] \), where:
\begin{equation*}
u = \frac{1}{s}, \quad v = \frac{1}{2} (2-t+t \sec q).
\end{equation*}

Then the coefficients of quadratic expression are

$$a_0=\frac{-6 t+(-4-4 s (-1+t)+6 t) \cos q+\sqrt{2} \sqrt{b_0}}{4 s (-1+t)}$$
where $$b_0=\left(1+s^2 (-1+t)\right) t \left(-4 (-2+t) \cos q+t \left(3+\cos (2q)\right)\right)$$
$$a_1=\frac{2 \left(s (-2+t)+2^{1-\frac{p}{2}} \left(s \left(2-t+t \sec q\right)\right)^{p/2}\right) \sin q-2 s t \tan q}{p s \left(2-t+t \sec q\right)}$$

$$a_2=\frac{ (p-2) s^{-1+p} \left(1+\cos q\right)\tau +p \left(p \left(-1+\cos q\right)+2 \sin q\right)}{p^2}$$
and $$\tau = 2^{2-p}\left(2-t+t \sec q\right)^{-2+p}.$$

Then
$$a_1^2 p^2 s (1-t)=4 s (1-t) \left(-1+2^{1-\frac{p}{2}} \left(s \left(2-t+t \sec q\right)\right)^{\frac{1}{2} (p-2)}\right)^2 \sin^2 q$$

and $X=p^2(1-t)s (a_1^2-4 a_0a_2)$ is equal to

\begin{equation*}
\begin{split}
&\left(
    -6t
    + (-4 - 4s(-1 + t) + 6t) \cos\left({q}\right) \right. \\
&\left. \quad + \sqrt{2} \sqrt{
        (1 + s^2(-1 + t)) t
        \left(
            -4(-2 + t) \cos\left({q}\right)
            + t \left(3 + \cos\left(\frac{2\pi}{p}\right)\right)
        \right)
    }
\right) \\
&\quad \times \left(
    2^{2 - p} (p-2) s^{-1 + p} \left(1 + \cos\left({q}\right)\right)
    \left(2 - t + t \sec\left({q}\right)\right)^{-2 + p} \right. \\
&\left. \quad + p \left(p(-1 + \cos\left({q}\right)) + 2 \sin\left({q}\right)\right)
\right) \\
&\quad + 4 s (1-t) \left(-1+2^{1-\frac{p}{2}} \left(s \left(2-t+t \sec q\right)\right)^{\frac{1}{2} (p-2)}\right)^2 \sin^2 q.
\end{split}
\end{equation*}

Now we define \[\begin{split}g(s,t)=&
    -6t
    + (-4 - 4s(-1 + t) + 6t) \cos\left({q}\right)  \\
& \quad + \sqrt{2} \sqrt{
        (1 + s^2(-1 + t)) t
        \left(
            -4(-2 + t) \cos\left({q}\right)
            + t \left(3 + \cos\left(\frac{2\pi}{p}\right)\right)
        \right)}.\end{split}\]
  Then in view of the trivial inequality \[\begin{split}&\left(
    2^{2 - p} (p-2) s^{-1 + p} \left(1 + \cos\left({q}\right)\right)
    \left(2 - t + t \sec\left({q}\right)\right)^{-2 + p} \right. \\
& \quad + p \left(p\left(-1 + \cos\left({q}\right)\right) + 2 \sin\left({q}\right)\right)\ge p \left(p\left(-1 + \cos\left({q}\right)\right) + 2 \sin\left({q}\right)\right),\end{split}\] $X\le 0$ if  $$A(s,t):=s \left(-1+2^{1-\frac{p}{2}} \left(s \left(2-t+t \sec q\right)\right)^{\frac{1}{2} (p-2)}\right)^2$$ $$+p g(s,t) \csc q\left(2-p \tan \frac{q}{2}\right)\le 0.$$
Let $$H(s):=s\left(-1+2^{1-\frac{p}{2}} \left(s \left(2-t+t \sec q\right)\right)^{\frac{1}{2} (p-2)}\right)^2.$$ Then $$H(s)= s(T(s)-1)^2,$$  where

$$T(s)  =2^{1-\frac{p}{2}} \left(s \left(2-t+t \sec q\right)\right)^{\frac{1}{2} (p-2)}.$$

Then $$H'(s)= (-1+T(s)) (-1+T(s)+(p-2) T(s)),$$ and the only solutions $s_\circ$ and $\tilde s$ of $H'(s)=0$ satisfy that $T(s_\circ)=1$ and $T(\tilde s) =\frac{1}{p-1}$.
Then the solutions to the following equation  $\partial _s(H(s,t))=0$ are $s_\circ$ and $\tilde s$, where  $$s_\circ=\frac{2 \cos q}{t+(2-t) \cos q},$$ and  $$\tilde s=s_\circ { (p-1)^{\frac{2}{2-p}}}.$$

Then the case $u<v$ is equivalent with $s>s_\circ$. Note that $\tilde s$ is not in this interval $[s_\circ,1]$.

Now $$\partial_s g(s,t)= (t-1) \left(-4 \cos q+X\right),$$ where $$X=\frac{\sqrt{2} s \sqrt{\left(1+s^2 (t-1)\right) t \left(4 (2-t) \cos q+t \left(3+\cos (2q)\right)\right)}}{1+s^2 (t-1)}.$$

And the only zero of $\partial_s g(s,t)=0$ is $$s=s_\circ=\frac{2 \cos q}{t+(2-t) \cos q}.$$
Then $$g(s_\circ,t)= 4 t \left(-1+\cos q\right)\le 0$$

Since $$g(1,t)-g(s_\circ,t)=-2 t-2 t \cos q+\sqrt{2} \sqrt{t^2 \left(-4 (-2+t) \cos q+t \left(3+\cos (2q)\right)\right)},$$ we have $g(1,t)-g(s_\circ,t)\le 0$ which is equivalent with the trivial inequality $$(-1+t) t^2 \sin^2 \frac{q}{2}\le 0$$

and \[\begin{split}g(0,t)-g(s_\circ,t)&=-2 t+2 (-2+t) \cos q\\&+\sqrt{2} \sqrt{t \left(-4 (-2+t) \cos q+t \left(3+\cos (2q)\right)\right)}\le 0\end{split}\]
 which is equivalent with the inequality $(-1+t) \cos^2 q<0$.

Hence  the maximum value of of $g(s,t)$ is $$ g(s_\circ, t)=4 t \left(-1+\cos q\right).$$

Thus $$A(s,t)\le B(s,t)=s \left(-1+2^{1-\frac{p}{2}} \left(s \left(2-t+t \sec q\right)\right)^{\frac{1}{2} (p-2)}\right)^2$$ $$+8 p t \csc q\sin^2 \frac{q}{2} \left(-2+p \tan \frac{q}{2}\right).$$

Since $\partial_s B(s,t)|_{s=s_\circ}=0$, and $B(s_\circ,t)<0$, it remains to show that $B(1,t)\le 0$.

Then for $$g(t)=2^{1-\frac{p}{2}} \left(2-t+t \sec q\right)^{\frac{1}{2} (p-2)}$$ we need to prove that $$(-1+g(t))^2+8 p t \csc q\sin^2 \frac{q}{2} \left(-2+p \tan \frac{q}{2}\right)\le 0.$$

To prove
$$
\left(-1+2^{1-\frac{p}{2}} \left(2-t+t \sec q\right)^{\frac{1}{2} (p-2)}\right)^2+8 p t \csc q\sin^2 \frac{q}{2} \left(-2+p \tan \frac{q}{2}\right)\le 0,$$ let us chose the change  $$\tau=2^{1-\frac{p}{2}} \left(2-t+t \sec q\right)^{\frac{1}{2} (p-2)}.$$  Then we come to the inequality $$H(\tau)=(-1+\tau)^2+8 p \left(\frac{-2+\left(2^{-1+\frac{p}{2}} \tau\right)^{\frac{2}{p-2}}}{-1+\sec q}\right) \csc q\sin^2 \frac{q}{2} \left(-2+p \tan \frac{q}{2}\right)\le 0$$ for $1\le \tau\le 2^{1-\frac{p}{2}} \left(1+\sec q\right)^{\frac{1}{2} (p-2)}$.

We compute the derivative:
\[
H'(\tau) = 2(-1 + \tau)
+ \frac{2^{3 + \frac{p}{2}} p \left(2^{-1 + \frac{p}{2}} \tau\right)^{-1 + \frac{2}{p-2}} \csc q \, \sin^2\left(\frac{q}{2}\right) \left(-2 + p\tan \frac{q}{2}\right)}{(p-2)(-1 + \sec q)}.
\]

The second derivative is:
\[
H''(\tau) = 2
+ \frac{2^{2 + p} \left(-1 + \frac{2}{p-2}\right) p \left(2^{-1 + \frac{p}{2}} \tau\right)^{-2 + \frac{2}{p-2}} \csc q \, \sin^2\left(\frac{q}{2}\right) \left(-2 + p\tan \frac{q}{2}\right)}{(p-2)(-1 + \sec q)}.
\]

Setting \( H''(\tau) = 0 \), we solve for the critical point \( \tau = \tau' \):
\[
\tau' = 2^{1 - \frac{p}{2}} \left( \frac{2^{-1 - p} (p-2)^2 \csc^2\left(\frac{\pi}{2p}\right) (-1 + \sec q) \sin q}{(-4 + p)p\left(-2 + p\tan \frac{q}{2}\right)} \right)^{\frac{p-2 }{6 - 2p}}.
\]

Evaluating the derivative at \( \tau' \), we obtain:
\[
H'(\tau') = -2
+ \frac{2^{3 - \frac{p}{2}} (-3 + p) \left( \frac{2^{-p} (p-2)^2 \tan q}{(-4 + p)p \left(-2 + p\tan \frac{q}{2}\right)} \right)^{\frac{-2 + p}{2(3 - p)}}}{-4 + p},
\]
which is negative. For \( p > 3 \), this is immediate. For \( 2 \le p \le 3 \), we use the inequalities:
\[
\frac{2^{-p} (p-2)^2 \tan q}{(-4 + p)p \left(-2 + p\tan \frac{q}{2}\right)} \le 1 \le \frac{2^{-2 + \frac{p}{2}} (-4 + p)}{-3 + p}.
\]

Furthermore, note that
\[
H'(1) = \frac{32p \csc q \, \sin^2\left(\frac{q}{2}\right) \left(-2 + p\tan \frac{q}{2}\right)}{(p-2)(-1 + \sec q)} \le 0.
\]

We also compute:
\begin{equation*}\begin{split}
&H'\left(2^{1 - \frac{p}{2}} \left(1 + \sec q\right)^{\frac{1}{2}(p-2)}\right)
\\&= -2
+ 2^{2 - \frac{p}{2}} \left(1 + \sec q\right)^{-1 + \frac{p}{2}}
\\&+ \frac{2^{3 + \frac{p}{2}} p \csc q \left(1 + \sec q\right)^{2 - \frac{p}{2}} \sin^2\left(\frac{q}{2}\right) \left(-2 + p\tan \frac{q}{2}\right)}{(p-2)(-1 + \sec q)}
\end{split}
\end{equation*}
and prove \begin{equation}\label{finain}
H'\left(2^{1 - \frac{p}{2}} \left(1 + \sec q\right)^{\frac{1}{2}(p-2)}\right)\le 0.\end{equation}

Having proved \eqref{finain}, we will conclude that \( H \) is decreasing. In particular, we have
\[
H(\tau) \le H(1) = 0.
\]
To prove the inequality \eqref{finain}, we begin with the estimate:
\[
(p-2)\left(-1 + \sec q\right) \le \frac{4}{\pi}.
\]
Next, prove that
\[
2 - 2^{2 - \frac{p}{2}} \left(1 + \sec q\right)^{-1 + \frac{p}{2}}
\ge
-\frac{1}{2} \cdot 2^{2 - \frac{p}{2}} \left(1 + \sec q\right)^{-1 + \frac{p}{2}}.
\]
The last inequality can be rewritten as $$\cos q >\frac{1}{-1+2^{\frac{p}{p-2}}}.$$ This inequality follows from $$\cos q\ge \frac{p-2}{p}$$ and $(1+1)^x\ge 1+x$ for $x=\frac{p}{p-2}$. Recall that $q=p/\pi$.

Thus, inequality \eqref{finain} will be implied by:
\[
- \left(1 + \sec q\right)^{-1 + \frac{p}{2}}
+ 2^p p \pi \csc q \left(1 + \sec q\right)^{2 - \frac{p}{2}} \sin^2\left(\frac{q}{2}\right)
\left(2 - p\tan \frac{q}{2}\right) \ge 0.
\]
Further simplifying, we get:
\[
-1 + 2^{1 + p} p \pi \csc q \left(1 + \sec q\right)^{3 - p} \sin^2\left(\frac{q}{2}\right)
\left(2 - p\tan \frac{q}{2}\right) \ge 0,
\]
and finally:
\[
-1 + 2^p p \pi \left(1 + \sec q\right)^{3 - p}\tan \frac{q}{2}
\left(2 - p\tan \frac{q}{2}\right) \ge 0.
\]
Since
\[
(1 + \sec q)(p - 2) \ge \frac{4}{\pi},
\]
it suffices to prove that
\[
-1 + 2^{6 - p} (p - 2)^{p - 3} p \pi^{p - 2}\tan \frac{q}{2}
\left(2 - p\tan \frac{q}{2}\right) \ge 0.
\]
Now, we use the inequality \(x^x \ge e^{-1/e} \approx 0.69\) and
\[
L(p): = \frac{p^2 \tan\frac{\pi }{2 p} \left(2-p \tan\frac{\pi }{2 p}\right)}{-2+p}
\ge L(\infty) = \frac{1}{4} (4-\pi ) \pi.
\]

The last inequality follows from the double inequality $$ \frac{\pi }{2 p}\le \tan\frac{\pi }{2 p}\le \frac{4-\pi }{p^2}+\frac{\pi }{2 p}.$$
Thus our inequality reduces to  the inequality:
\[
2^{4-p} e^{-1/e} (4-\pi ) \pi ^{-1+p}>8 e^{-1/e} \left(1+(-1+p) \left(-1+\frac{\pi }{2}\right)\right) (4-\pi )>p,
\]
which holds for every $p\ge 2$.

\end{proof}

\subsection{Proof of the case $p\ge 4$}

We have already established the inequality~\eqref{bein}, and now aim to prove it with the function \( h(t) \) in place of the function \( -\cos\left(\frac{p}{2}(\pi - t)\right) \). Observe that \( h(t) \) coincides with \( -\cos\left(\frac{p}{2}(\pi - t)\right) \) on the interval \( t \in \left[\pi - \frac{2\pi}{p}, \pi\right] \). Therefore, the inequality holds on this interval. Moreover, since
\[
h\left(\left[\pi - \frac{2\pi}{p}, \pi\right]\right) = [-1, 1],
\]
for any \( t \in \left[0, \pi - \frac{2\pi}{p}\right] \), there exists some \( t' \in \left[\pi - \frac{2\pi}{p}, \pi\right] \) such that
\[
h(t) = h(t') = -\cos\left(\frac{p}{2}(\pi - t')\right).
\]
Since the cosine function is decreasing on the interval \( [0, \pi] \), we have
\[
\left(\frac{1}{2r} + \frac{r}{2} + \cos t\right)^{p/2} \geq \left(\frac{1}{2r} + \frac{r}{2} + \cos t'\right)^{p/2}.
\]
This establishes the inequality for \( t \in \left[0, \pi - \frac{2\pi}{p}\right] \), completing the proof.
\subsection*{Acknowledgments} This research was supported by the Ministry of Education, Science and Innovation of Montenegro through the grant "Mathematical Analysis, Optimization and Machine Learning."

\end{document}